\documentclass[12pt]{article}
\usepackage{graphicx}
\usepackage{amssymb, amsthm, amsmath, amsfonts}
\usepackage{mathrsfs,microtype}
\usepackage{epstopdf, color}
\usepackage{tikz}

\usepackage{hyperref}

\usetikzlibrary{arrows,decorations.pathmorphing,decorations.footprints,fadings,calc,trees,mindmap,shadows,decorations.text,patterns,positioning,shapes,matrix,fit,through,decorations.pathreplacing}
\usepackage[margin=1in]{geometry}

\theoremstyle{plain}
\newtheorem{theorem}{Theorem}[section]
\newtheorem{lemma}[theorem]{Lemma}

\newtheorem{question}[theorem]{Question}

\theoremstyle{definition}

\theoremstyle{remark}

\newcommand{\drawClaw}[5]{
\begin{tikzpicture}[scale=.3, every node/.style={scale=.3}]
    \node[fill=#1] (v1) at (1.4,1.5) [circle,draw] {};
    \node[fill=#2] (v2) at (1.4,2.5) [circle,draw] {};
    \node[fill=#3] (v3) at (2,2) [circle,draw] {};
    \node[fill=#4] (v4) at (3,2) [circle,draw] {};
    \node[fill=#5] (v5) at (4,2) [circle,draw] {};
    \draw (v1) -- (v3);
    \draw (v2) -- (v3);
    \draw (v3) -- (v4);
    \draw (v4) -- (v5);
\end{tikzpicture}}

\title{Reversible peg solitaire on graphs}
\author{John Engbers \and Christopher Stocker}
\date{\today\thanks{\{john.engbers, christopher.stocker\}@marquette.edu; Department of Mathematics, Statistics and Computer Science, Marquette University, Milwaukee, WI 53201} }

\begin{document}

\maketitle

\begin{abstract}
The game of peg solitaire on graphs was introduced by Beeler and Hoilman in 2011.  In this game, pegs are initially placed on all but one vertex of a graph $G$.  If $xyz$ forms a path in $G$ and there are pegs on vertices $x$ and $y$ but not $z$, then a {\em jump} places a peg on $z$ and removes the pegs from $x$ and $y$.  A graph is called solvable if, for some configuration of pegs occupying all but one vertex, some sequence of jumps leaves a single peg. We study the game of {\em reversible peg solitaire}, where there are again initially pegs on all but one vertex, but now both jumps and unjumps (the reversal of a jump) are allowed. We show that in this game all non-star graphs that contain a vertex of degree at least three are solvable, that cycles and paths on $n$ vertices, where $n$ is divisible by $2$ or $3$, are solvable, and that all other graphs are not solvable. We also classify the possible starting hole and ending peg positions for solvable graphs.
\end{abstract}

\section{Introduction}

Peg solitaire is a game on geometric boards that has been recently generalized to connected simple graphs by Beeler and Hoilman \cite{BeelerHoilman1}.  In the game of peg solitaire on graphs, all vertices but one start occupied by a \emph{peg} (the vertex without a peg is said to have a \emph{hole}).  If $x$ and $y$ are adjacent and $y$ and $z$ are adjacent with pegs on $x$ and $y$ and a hole on $z$, then the legal move is to \emph{jump} the peg on $x$ over the peg on $y$ into the hole on $z$ while removing the peg on $y$.  See Figure \ref{fig-peg solitaire}.

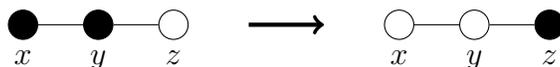
\begin{figure}[ht!]
\centering
\begin{tikzpicture}
    \node (v1) at (1,1) [circle,draw,fill,label=270:$x$] {};
    \node (v2) at (2,1) [circle,draw,fill,label=270:$y$] {};
    \node (v3) at (3,1) [circle,draw,label=270:$z$] {};
    \node (v4) at (6,1) [circle,draw,label=270:$x$] {};
    \node (v5) at (7,1) [circle,draw,label=270:$y$] {};
    \node (v6) at (8,1) [circle,draw,fill,label=270:$z$] {};
    
    \draw [->,ultra thick] (4,1) -- (5,1);
    
    \draw (v1) -- (v2);
    \draw (v2) -- (v3);
    \draw (v4) -- (v5);
    \draw (v5) -- (v6);
\end{tikzpicture}
\caption{A jump in peg solitaire.}
\label{fig-peg solitaire}
\end{figure}

The purpose of the game is to reduce the total number of pegs on the vertices to one.  If this is achieved from some starting configuration with exactly one hole the graph is said to be \emph{solvable}.  If it is achievable from every starting configuration with exactly one hole the graph is said to be \emph{freely solvable}. 
Clearly, solvability requires the graph to be connected. Several results on which graphs are freely solvable, which are solvable but not freely solvable, and which are not solvable are given in \cite{BeelerHoilman1,BeelerHoilmanWindmillStar,Walvoort}. As a sample result, in \cite{BeelerHoilmanWindmillStar} the game is played on the double star $DS(L,R)$ (with $L \geq R$), which is the graph consisting of a fixed edge $uv$ that has $L$ pendant edges joined to $u$ and $R$ pendant edges joined to $v$. Beeler and Hoilman show that $DS(L,R)$ is freely solvable if and only if $L=R$ and $R \neq 1$, and 
solvable if and only if $L \leq R+1$. Fully characterizing the connected graphs $G$ that are freely solvable or solvable for peg solitaire on graphs seems to be a difficult question. It is also worth mentioning here that there are many interesting results and techniques related to traditional peg solitaire on geometric boards, see \cite{Beasley,WinningWays,DezaOnn}.  

Variations of peg solitaire on graphs have recently been introduced.  One such variation, called fool's solitaire, asks for the maximum number of pegs that can be left on vertices where no possible jumps remain, see  \cite{BeelerRodriguez,LoebWise}. In this paper, we introduce another variation, which is mentioned with regards to algebraic techniques used to study traditional peg solitaire in \cite{WinningWays}.  We consider the game of {\em reversible peg solitaire}, which allows not only jumps but also the additional move of an \emph{unjump}, which is the reversal of a jump.  Specifically, if $x$ and $y$ are adjacent and $y$ and $z$ are adjacent with holes on $x$ and $y$ and a peg on $z$, then a second legal move is to unjump the peg from $z$ to $x$, creating a peg on $y$.  We can also view this as allowing the hole on $x$ to jump over the hole on $y$, creating a hole on $z$ (but also creating pegs on $x$ and $y$).  See Figure \ref{fig-peg solitaire2}.

\begin{figure}[ht!]
\centering
\begin{tikzpicture}
    \node (v1) at (6,1) [circle,draw,fill,label=270:$x$] {};
    \node (v2) at (7,1) [circle,draw,fill,label=270:$y$] {};
    \node (v3) at (8,1) [circle,draw,label=270:$z$] {};
    \node (v4) at (1,1) [circle,draw,label=270:$x$] {};
    \node (v5) at (2,1) [circle,draw,label=270:$y$] {};
    \node (v6) at (3,1) [circle,draw,fill,label=270:$z$] {};
    
    \draw [->,ultra thick] (4,1) -- (5,1);
    
    \draw (v1) -- (v2);
    \draw (v2) -- (v3);
    \draw (v4) -- (v5);
    \draw (v5) -- (v6);
\end{tikzpicture}
\caption{An unjump in reversible peg solitaire.}
\label{fig-peg solitaire2}
\end{figure}
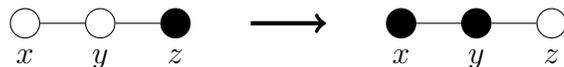

Notice that reversible peg solitaire may also be viewed as a restricted version of Lights Out on graphs, a game where a closed neighborhood may flip all states (here pegs/holes).  In this formulation, we are allowed to flip the states of all vertices in a path on three vertices (instead of an entire closed neighborhood) if the vertices are in one of two starting configurations. In particular, any path on three vertices whose endpoints contain exactly one peg may be flipped. For a survey of Lights Out, see \cite{FleischerYu}.

Following the terminology for peg solitaire, we say that a graph is {\em solvable} for reversible peg solitaire if for some starting configuration with exactly one hole, some combination of jumps and unjumps eventually produces a configuration with a single peg. If a graph is solvable for every starting configuration with exactly one hole, we say that the graph is \emph{freely solvable}. If we may freely choose the location of the final peg in a freely solvable graph, we say that the graph is \emph{doubly freely solvable}. Note that any graph that is doubly freely solvable is freely solvable, and any graph that is freely solvable is solvable.

For reversible peg solitaire, which graphs are solvable? Clearly anything (freely/doubly freely) solvable in peg solitaire (which uses only jumps) will also be (freely/doubly freely) solvable in reversible peg solitare, but some graphs that are not solvable in peg solitaire may now become solvable in reversible peg solitaire. Our first result shows, however, that not \emph{all} graphs are solvable in reversible peg solitare.

\begin{theorem} \label{thm-star}
For any $n \geq 4$, the star $K_{1,n-1}$ is not solvable.
\end{theorem}

\begin{proof}
Notice that both jumps and unjumps on the star require there to be a peg on one leaf and a hole on a second leaf; this shows that to make an initial move the hole must start on a leaf.  Given this, each move will preserve the total number of pegs on the leaves and toggle the center between having a peg and being a hole.  Therefore the only possible configurations on $K_{1,n-1}$ have either $n-1$ or $n-2$ pegs, so for $n \geq 4$ there will never be a single peg remaining.
\end{proof}

As the following theorem shows, with unjumps allowed \emph{most} graphs are freely solvable. 

\begin{theorem}\label{thm-main}
Let $G$ be a connected graph on $n$ vertices.  If $G \neq K_{1,n-1}$ and $G$ has a vertex of degree at least $3$, then $G$ is freely solvable. Furthermore $G$ is doubly freely solvable if and only if there is a path joining two vertices of degree at least $3$ whose length is not divisible by $3$.  

\end{theorem}

In the proof of Theorem \ref{thm-main}, which is given in Section \ref{sec-proof}, we provide all possible starting hole and ending peg positions for all freely solvable but not doubly freely solvable graphs containing a vertex of degree at least $3$. 

The only connected graphs that are not covered by Theorems \ref{thm-star} and \ref{thm-main} are paths and cycles. 
Let $P_n$ and $C_n$ denote a path and a cycle on $n$ vertices, respectively. 
We also have the following, which we prove in Section \ref{sec-proof}.

\begin{theorem}\label{thm-paths and cycles}
Let $n \geq 2$ be an integer. 
\begin{enumerate}
	\item If $n$ is not divisible by $2$ or $3$, then $P_n$ and $C_n$ are not solvable. 
	\item If $n$ is divisible by $3$, then $P_n$ is solvable but not freely solvable, and $C_n$ is freely solvable but not doubly freely solvable.  
	\item If $n$ is not divisible by $3$ but is divisible by $2$, then $P_n$ is solvable but not freely solvable, and $C_n$ is doubly freely solvable
\end{enumerate}

\end{theorem}

In the proof of Theorem \ref{thm-paths and cycles}, which is given in Section \ref{sec-proof}, we provide all possible starting hole and ending peg positions for all solvable paths and cycles.  



We also mention two natural questions. 
Any graph that is solvable in peg solitaire on graphs is also solvable in reversible peg solitaire on graphs, and in particular can be solved with zero unjumps.  Suppose that we let $k$ count the minimum number of unjumps needed to solve a graph $G$ in reversible peg solitaire. There are two natural questions associated with this parameter $k$.  
\begin{question}
Given a graph $G$ that is solvable in reversible peg solitaire, what is the minimum number of unjumps necessary to solve $G$? 
\end{question}

\begin{question}
For a fixed $k$, which graphs are solvable with at most $k$ unjumps? 
\end{question}

The proof of Theorem \ref{thm-main} uses at most $cn^2$ unjumps to solve a connected non-star graph containing a vertex of degree at least 3.

\section{Proofs}\label{sec-proof}

In this section we will first prove that if $G \neq K_{1,n-1}$ is a graph that contains a vertex of degree at least $3$, then $G$ is solvable. If $n = 4$, then since there is a vertex of degree at least $3$ and $G \neq K_{1,3}$, there must exist a triangle with a pendant edge as a (not necessarily induced) subgraph.  Using only the edges on this subgraph, $G$ is solvable by inspection.  Therefore we assume that $n \geq 5$; note that we may also assume that $G$ is a non-star tree with a vertex of degree at least $3$.

The main idea of the proof is to analyze the configurations of pegs on a graph $H$, where $H$ is a claw $K_{1,3}$ with one subdivided edge (see Figure \ref{fig-subdivided claw}). Notice that any connected non-star graph with $n \geq 5$ and a vertex of degree at least $3$ has $H$ as a (not necessarily induced) subgraph.  Using these configurations, we show how to iteratively bring pegs from outside $H$ into $H$ and remove them, which eventually removes all pegs but one.

\begin{figure}[ht!]
\centering
\begin{tikzpicture}[scale=1.2]
    \node (v1) at (1.4,1.35) [circle,draw,label=270:$a$] {};
    \node (v2) at (1.4,2.65) [circle,draw,label=90:$b$] {};
    \node (v3) at (2,2) [circle,draw,label=270:$c$] {};
    \node (v4) at (3,2) [circle,draw,label=270:$d$] {};
    \node (v5) at (4,2) [circle,draw,label=270:$e$] {};
    
    \draw (v1) -- (v3);
    \draw (v2) -- (v3);
    \draw (v3) -- (v4);
    \draw (v4) -- (v5);
\end{tikzpicture}
\caption{The graph $H$, which is a claw with one subdivided edge.}
\label{fig-subdivided claw}
\end{figure}
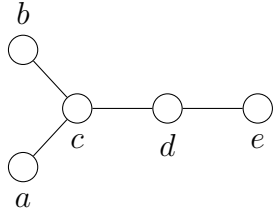

In the written notation for a configuration of pegs on $H$, we will let a letter indicate the presence of a peg and the absence of a letter indicate a hole; for example, $ace$ indicates pegs on $a$, $c$, and $e$ and holes on $b$ and $d$. 
Two configurations on $H$ are \emph{equivalent} if we may move from one configuration to the other through a series of jumps and unjumps within $H$. We now separate out the configurations on $H$ that are equivalent to a configuration on $H$ that contain a single peg. 

\begin{lemma}\label{lem-classes} 
The rows in the following table represent two equivalence classes of configurations on $H$. 
The four unlisted configurations contain no jumps or unjumps within $H$.
\begin{center}
\begin{tabular}{c|l}
{\rm Class A} \,&\, a, b, d, e, ac, bc, cd, abe, ade, bde, abcd, abce, acde, bcde\\
\hline
{\rm Class B} \,&\, c, ab, ad, ae, bd, be, de, abc, acd, ace, bcd,  bce, cde, abde\\
\end{tabular}
\end{center}
\end{lemma}

\begin{proof}
Letting $x\leftrightarrow y$ indicate that configurations $x$ and $y$ differ by a single jump/unjump, we have:
\[
e \leftrightarrow cd \leftrightarrow a \leftrightarrow bc \leftrightarrow d \leftrightarrow ac \leftrightarrow ade \leftrightarrow bcde \leftrightarrow abe \leftrightarrow acde \leftrightarrow bde \leftrightarrow abce.
\]
Noting that also $b \leftrightarrow cd$ and $abcd \leftrightarrow abe$, this produces the configurations in Class A.  For Class B, we have
\[
c \leftrightarrow de \leftrightarrow bce \leftrightarrow ae \leftrightarrow acd \leftrightarrow ab \leftrightarrow bcd \leftrightarrow be \leftrightarrow ace.
\]
Noting that also $ad \leftrightarrow abc \leftrightarrow bd \leftrightarrow acd$, $cde \leftrightarrow ae$,  
 and $abde \leftrightarrow abc$, this produces the configurations in Class B.  The remaining four configurations ($abcde$, $abd$, $ce$, and the empty configuration) each have no possible jumps or unjumps within $H$.
\end{proof}

Next, we define a move that we will repeatedly use in our proof. 

\begin{eqnarray*}
\text{\textbf{$P_4$--Move:}} && \text{Given a $P_4$ that has one peg (hole, resp.) on an endpoint } \\
&&\text{and holes (pegs, resp.) on the other three vertices, we can }\\
&&\text{move the peg (hole, resp.) to the other endpoint.} \\ 
\end{eqnarray*}
\begin{proof}
Do an unjump (jump, resp.) to the two middle pegs, and then do a jump (unjump, resp.) to the other endpoint.
\end{proof}

Now we move on to the main proof.

\begin{proof}[Proof of Theorem \ref{thm-main}]

Let $G$ be a connected non-star graph with $n\geq 5$ and a vertex of degree at least $3$. We first show that $G$ is freely solvable, and discuss the necessary and sufficient conditions for $G$ to be doubly freely solvable at the end. 
We can find $H$ as a (not necessarily induced) subgraph of $G$; fix one such $H$ for the remainder of the freely solvable proof. Suppose also that there are are pegs on all but a single vertex in $G$.  If the hole starts outside of $H$, we can use the $P_4$--Move to shift it onto $H$. 
In particular, we now have a configuration on $H$ in Class A or Class B.  The following procedure will leave a single peg on a vertex in $G$.


Fix a peg outside of $H$ that is closest to $H$.  We move to an equivalent configuration on $H$ (within the same class, as defined in Lemma \ref{lem-classes}); the configuration chosen will depend on the distance from $H$ as well as the vertex in $H$ closest to the peg under consideration. Since this is a closest peg, we use the $P_4$--Move, if necessary, to move the peg within distance $3$ of $H$. See Figure \ref{fig-eliminate a}; in this case we consider the peg to be on either $x_1$, $x_2$, or $x_3$.

We now absorb the peg into $H$ while maintaining a configuration on $H$ in either Class A or Class B applying one of the following cases.  

\medskip

\textbf{Case 1: The peg is nearest to $a$.} See Figure \ref{fig-eliminate a}. (Notice that this is equivalent to the peg being nearest to $b$.)

\begin{figure}[ht]
\centering
\begin{tikzpicture}[scale=1]
    \node (v1) at (1.4,1.35) [circle,draw,label=270:$a$] {};
    \node (v2) at (1.4,2.65) [circle,draw,label=90:$b$] {};
    \node (v3) at (2,2) [circle,draw,label=270:$c$] {};
    \node (v4) at (3,2) [circle,draw,label=270:$d$] {};
    \node (v5) at (4,2) [circle,draw,label=270:$e$] {};
    
    \draw (v1) -- (v3);
    \draw (v2) -- (v3);
    \draw (v3) -- (v4);
    \draw (v4) -- (v5);
    
    \node (v6) at (0.4, 1.35) [circle,draw,label=270:$x_1$] {};
    \node (v7) at (-0.6, 1.35) [circle,draw,label=270:$x_2$] {};
    \node (v8) at (-1.6, 1.35) [circle,draw,label=270:$x_3$] {};
    \draw (v6) -- (v1);
    \draw (v7) -- (v6);
    \draw (v8) -- (v7);
\end{tikzpicture}
\caption{Case 1.}
\label{fig-eliminate a}
\end{figure}

If the configuration on $H$ is in Class A to start, we use configuration $b$ (\drawClaw{white}{black}{white}{white}{white}).  Using the $P_4$--Move we can move the peg to either $a$, $c$, or $d$, which puts us in either Class A or Class B.

If the configuration on $H$ is in Class B to start, we use configuration $de$ (\drawClaw{white}{white}{white}{black}{black}).  Using the $P_4$--Move we can move the peg to $a$, $b$, or $c$, which puts us in either Class A or Class B.

\medskip

\textbf{Case 2: The peg is nearest to $c$.} 

    %
    %

If the configuration on $H$ is in Class A to start, we use configuration $b$ (\drawClaw{white}{black}{white}{white}{white}).  Using the $P_4$--Move we can move the peg to $c$, $d$, or $e$, which puts us in either Class A or Class B.

If the configuration on $H$ is in Class B to start and we consider a peg on $x_1$ or $x_3$, then we use configuration $ab$ (\drawClaw{black}{black}{white}{white}{white}).  Using the $P_4$--Move we can move the peg to $c$ or $e$, which puts us in either Class A or Class B.

If instead the configuration on $H$ is in Class B to start and we consider a peg on $x_2$, then we use configuration $be$ (\drawClaw{white}{black}{white}{white}{black}).  Using the $P_4$--Move we can move the peg to $a$, which puts us in Class A.

\medskip

\textbf{Case 3: The peg is nearest to $d$.} 

    %
    %

If the configuration on $H$ is in Class A to start, we use configuration $b$ (\drawClaw{white}{black}{white}{white}{white}).  Using the $P_4$--Move we can move the peg to $a$, $c$, or $d$, which puts us in either Class A or Class B.

If the configuration on $H$ is in Class B to start, then we use configuration $be$ (\drawClaw{white}{black}{white}{white}{black}).  Using the $P_4$--Move we can move the peg to $a$, $c$, or $d$, which puts us in either Class A or Class B.

\medskip

\textbf{Case 4: The peg is nearest to $e$.} 

    %
    %

If the configuration on $H$ is in Class A to start, we use configuration $b$ (\drawClaw{white}{black}{white}{white}{white}).  Using the $P_4$--Move we can move the peg to $c$, $d$, or $e$, which puts us in either Class A or Class B.

If the configuration on $H$ is in Class B to start and we consider a peg on $x_1$ or $x_3$, then we use configuration $ab$ (\drawClaw{black}{black}{white}{white}{white}).  Using the $P_4$--Move we can move the peg to $c$ or $e$, which puts us in either Class A or Class B.

If instead the configuration on $H$ is in Class B to start and we consider a peg on $x_2$, then we use configuration $c$ (\drawClaw{white}{white}{black}{white}{white}).  Using the $P_4$--Move we can move the peg to $d$, which puts us in Class A.

\medskip  

Since each step reduces the number of pegs outside $H$ by one, after iterating $|V(G)|-5$ times the process terminates with $H$ in Class A or Class B and no pegs outside of $H$.  Since each of Class A and Class B contains a configuration with a single peg, the proof is complete.

\medskip

We next prove necessary and sufficient conditions for $G$ to be doubly freely solvable. First, suppose that all paths joining vertices of degree at least $3$ have length divisible by $3$. We use a weighting argument to show that $G$ is not doubly freely solvable. Choose a vertex $v$ of degree at least $3$, and assign a weight of $0$ to all vertices $w$ such that a path from $v$ to $w$ with length divisible by $3$ exists.  All other vertices are assigned a weight of $1$. 
This is well-defined by the assumption on the vertices of degree at least $3$; note that all vertices of degree at least $3$ receive weight $0$.  Then define the \emph{total weight of a configuration} to be the sum (mod $2$) of the weights on the vertices containing pegs (this is similar to a pagoda function defined in \cite{WinningWays}).  Since every path $P_3$ contains exactly two vertices with weight $1$, each jump and unjump preserves the total weight. This implies that initial configurations with total weight $0$ must end with a peg on a vertex having weight $0$, and initial configurations with total weight $1$ must end with a peg on a vertex having weight $1$.  The $P_4$--Move shows that a single peg on any weight $0$ vertex can be moved to any other weight $0$ vertex.  The $P_4$--Move also shows that a single peg on any weight $1$ vertex can be moved to either $a$, $b$, $d$, or $e$ in $H$, and since these configurations are equivalent in $H$, this shows that a single peg on any weight $1$ vertex can be moved to any other weight $1$ vertex.

Now suppose that there are two vertices $v_1$ and $v_2$ of degree at least $3$ that have a path $P$ between them of length not divisible by $3$. Consider the two possible copies of $H$, $H_1$ and $H_2$, with degree $3$ vertices $v_1$ and $v_2$ so that the respective $d$ and $e$ vertices (in the respective copies of $H$) lie on $P$; if $P$ only contains $v_1$ and $v_2$, then we require that the respective $e$ vertices lie on the respective copies of $H$. We have shown that any initial hole can be reduced to a single peg on a vertex in $H_1$. By the $P_4$--Move, a peg on $v_1$ in $H_1$ can be moved to a peg on either $d$ or $e$ in $H_2$. By Lemma \ref{lem-classes} this peg can be moved to either $e$ or $d$, respectively, in $H_2$, which by the $P_4$--Move again can be moved back to $H_1$ to a vertex other than $v_1$. This procedure is reversible, and so by using $P$ and $H_2$ we can move from Class A to Class B in $H_1$ when there is a single peg remaining in $G$. Since any vertex in $G$ has a path to a vertex in $H_1$ so that the length of the path is a multiple of $3$, we may use the $P_4$--Move to place the final peg on any vertex of $G$.
\end{proof}


\begin{proof}[Proof of Theorem \ref{thm-paths and cycles}]
It is shown in \cite{BeelerHoilman1} that $P_{2k}$ is solvable for peg solitaire (without unjumps) when the hole starts on a vertex adjacent to a leaf, so it remains solvable in reverse peg solitaire.  
For $P_{3\ell}$ with $\ell$ odd, let the vertices be $\{1,2,3,\ldots,3\ell-1,3\ell\}$ and start with the hole on vertex $3$.  Jump from $1$ into $3$, and then use the $P_4$--Move to shift the hole on vertex $2$ to vertex $3\ell -1$.  Then vertices $2$, $3$, $\ldots$, $3\ell-1$, and $3\ell$ form an even path where there is a single hole on a vertex adjacent to a leaf, which is solvable. Note that by reversing the roles of pegs and holes in each of these cases, we obtain possible starting hole and ending peg configurations for solvable paths.

Since $P_n$ is solvable in these cases and $P_n$ is a subgraph of $C_n$, $C_n$ is freely solvable in these cases.

We now fully classify paths and cycles using another weighting argument.  
We use, as weights, the elements the multiplicative quaternion group $Q_8$, which has presentation $$Q_8 = \langle -1,i,j,k | (-1)^2 = 1, i^2=j^2=k^2=ijk=-1 \rangle.$$ In particular, note that $ij=k$, $jk = i$, and $ki = j$.

For $P_n$ with vertices $\{1,2,3,\ldots,n\}$, assign weight $i$ to  vertex $x$ if $x \mod 3 = 1$, weight $j$ to vertex $x$  if $x\mod 3 = 2$, and weight $k$ to vertex $x$ if $x\mod 3 = 0$.  The \emph{total weight of a configuration} is the product of the weights of the vertices containing pegs when written from smallest to largest vertex. For example, if on $P_5$ we have pegs on vertices $1$, $3$, $4$, and $5$, then the total weight of that configuration is $ikij = -i$.

Note that: (a) moves preserve the total weight of a configuration, (b) we can assume, by the $P_4$-move, an initial (final, resp.) configuration has a hole (peg, resp.) on vertex $1$, $2$, or $3$, and (c) the total weight of a configuration with a single peg is either $i$ (if the peg is on vertex $1$), $j$ (if the peg is on vertex $2$), or $k$ (if the peg is on vertex $3$).

We use these observations to fully classify paths $P_n$ where $n \geq 3$. We have the following six cases. 
\begin{enumerate}
\item If $n=6\ell$, then an initial hole on $1$, $2$, or $3$ gives an initial configuration weight of $-i$, $j$, or $-k$, respectively.  Therefore the hole must start on $2$ and the final peg must end on $2$.
\item If $n=6\ell+1$, then an initial hole on $1$, $2$, or $3$ gives an initial configuration weight of $1$, $-k$, or $-j$, respectively.  Therefore these paths are not solvable.
\item If $n=6\ell+2$, then an initial hole on $1$, $2$, or $3$ gives an initial configuration weight of $j$, $i$, or $1$, respectively.  Therefore the hole must start on $1$ or $2$ and the final peg must end on $2$ or $1$, respectively.
\item If $n=6\ell+3$, then an initial hole on $1$, $2$, or $3$ gives an initial configuration weight of $i$, $-j$, or $k$, respectively.  Therefore the hole must start on $1$ or $3$ and and the final peg must end on $1$ or $3$, respectively.
\item If $n=6\ell+4$, then an initial hole on $1$, $2$, or $3$ gives an initial configuration weight of $-1$, $k$, or $j$, respectively.  Therefore the hole must start on $2$ or $3$ and and the final peg must end on $3$ or $2$, respectively.
\item If $n=6\ell+5$, then an initial hole on $1$, $2$, or $3$ gives an initial configuration weight of $-j$, $-i$, or $-1$, respectively.  Therefore these paths are not solvable.
\end{enumerate}
Since these six cases give the total weight of any possible initial hole or final peg configuration (and the $P_4$--Move allows an initial hole or final peg to be shifted by distance $3$), this fully classifies the possible starting hole and final peg positions for $P_n$.

What about $C_n$ for $n \geq 3$? We know that the cycle $C_n$ is freely solvable unless $n=6k+1$ or $n=6k+5$.  We use a similar weighting scheme to show the remaining cycles are not solvable and to classify the cycles that are doubly freely solvable.

Fix one cyclic orientation of the vertices of $C_n$; label them $\{1,2,\ldots,n\}$.  We then consider $C_{3n}$ with vertices $\{1,2,\ldots 3n\}$ such that each move on $C_n$ corresponds to \emph{three} moves on $C_{3n}$, where the three moves are equivalent copies of the move on those vertices with labels that differ by $n$.  For example, if $n=5$ and a move jumps a peg on vertex $2$ over a peg on vertex $1$ into a hole on vertex $5$, then in $C_{15}$ the pegs on vertices $2$, $7$, and $12$ jump over the pegs on vertices $1$, $6$, and $11$ into holes on vertices $15$, $5$, and $10$.

As before, define the \emph{total weight of a configuration} to be the product of the weights of the vertices containing pegs when written from smallest to largest vertex. Notice that the sets of moves made on $C_{3n}$ preserve the total weight (here, the fact that $ijk = kij = jki$ and $jik = kji = kij$ is essential).

Suppose that $n=6\ell+1$.  Then an initial configuration in $C_{3n}$ (corresponding to an initial configuration in $C_{n}$ with a single hole) has total weight $1$ and a final configuration in $C_{3n}$ (corresponding to a final configuration in $C_n$ with a single peg) has pegs in $C_{3n}$ on $x$, $x+n$, and $x+2n$ for some $x \in \{1,2,\ldots,n\}$.  But this means a final configuration in $C_{3n}$ that corresponds to a single peg in $C_n$ has total weight $-1$.  If $n=6\ell+5$, then a similar analysis shows that an initial configuration in $C_{3n}$ (corresponding to an initial configuration in $C_{n}$ with a single hole) has total weight $-1$ while a final configuration in $C_{3n}$ (corresponding to a final configuration in $C_{n}$ with a single peg) has total weight $1$.  Therefore $C_n$, where $n=6\ell+1$ or $6\ell+5$, is not solvable.


Similar arguments show that if a hole starts on vertex $x$ in $C_{6\ell}$ and $C_{6\ell+3}$, then the final peg must be on vertex $x + 3q$ for some integer $q$.  In $C_{6\ell+2}$ and $C_{6\ell+4}$, the initial hole and final peg can be anywhere (using the $P_4$--Move).
%
%
%
\end{proof}



\end{document}